\documentclass[12p]{amsart}
\usepackage{amscd,amsmath,amssymb,amsthm,amsfonts,epsfig,graphics}
\usepackage{geometry}
\usepackage{graphicx}
\usepackage{mathrsfs,amssymb}
\usepackage{bm}
\usepackage{subfigure}
\usepackage{color}

\theoremstyle{plain}

\newtheorem{corollary}{Corollary}

\newtheorem{definition}{Definition}

\newtheorem{lemma}{Lemma}

\newtheorem{proposition}{Proposition}

\newtheorem{theorem}{Theorem}
\numberwithin{equation}{section}
\newcommand{\bth}{\begin{theorem}}
	\newcommand{\ble}{\begin{lemma}}
		\newcommand{\bcor}{\begin{corr}}

			\newcommand{\bdeff}{\begin{deff}}

				\newcommand{\bprop}{\begin{proposition}}
					\newcommand{\ele}{\end{lemma}}
				\newcommand{\ecor}{\end{corr}}
			\newcommand{\edeff}{\end{deff}}
		
		\newcommand{\eprop}{\end{proposition}}

	\newcommand{\la}{\lambda}
	
	\newcommand{\eps}{\varepsilon}

	\newcommand{\supp}{\text{supp }}
	\renewcommand{\Pi}{\varPi}

	\renewcommand{\epsilon}{\varepsilon}

	\newcommand{\ls}{\lesssim}
	\newcommand{\M}{M}
	\newcommand{\N}{\Omega}
	\newcommand{\D}{\mathcal{D}}
	
	\newcommand{\gs}{\gtrsim}

	\newcommand{\lap}{\mbox{$\Delta_g$}}

\begin{document}
		\title[]{Sharp $L^p$ estimates and size of nodal sets for generalized Steklov eigenfunctions}
		
		\author{Xiaoqi Huang, Yannick Sire, Xing Wang and Cheng Zhang}

		
		\address{Department of Mathematics, University of Maryland, College Park, MD 20742, United States}
		\email{xhuang49@umd.edu}
		\address{Department of Mathematics, Johns Hopkins University, Baltimore, MD 21218, United States}
		\email{ysire1@jhu.edu}
		\address{Department of Mathematics, Hunan University, Changsha, HN 410012, China}
		\email{kwxing@gmail.com}
		\address{Mathematical Sciences Center, Tsinghua University, Beijing 100084, China}
		\email{czhang98@tsinghua.edu.cn}

		

		\dedicatory{}

		\begin{abstract}
			We prove sharp $L^p$ estimates for the Steklov eigenfunctions on compact manifolds with boundary in terms of their $L^2$ norms on the boundary. We prove it by  establishing $L^p$ bounds for the harmonic extension operators as well as the spectral projection operators on the boundary. Moreover, we derive  lower bounds on the size of nodal sets for a variation of the Steklov spectral problem. We consider a generalized version of the Steklov problem by adding a non-smooth potential on the boundary but some of our results are new even without  potential.  
		\end{abstract}
		
		\maketitle
		
		\section{Introduction}
		Eigenfunction estimates have been recently considered in the case of Schr\"odinger operators with singular potentials (see e.g. \cite{simon}, \cite{BSS19}, \cite{blair2021uniform}, \cite{jmpa}, \cite{fs}, \cite{hs}, \cite{hz}, \cite{hztorus}). In the present paper, we investigate a generalization of the well-known Steklov problem with non-smooth potentials. For surveys on the Steklov  problem, see e.g. \cite{GP17}, \cite{col}.
		
		Let $(\Omega,h)$ be a smooth manifold with boundary $(M,g)$, where dim $\Omega=n+1\ge2$ and $h|_M=g$. The Steklov eigenvalue problem with potential $V$ is
		\[\begin{cases}
			\Delta_h e_\la(x)=0,\ x\in\Omega\\
			\partial_ \nu e_\la(x)+V(x)e_\la(x)=\la e_\la(x),\ x\in \partial \Omega=M.
		\end{cases}\]
		Here $\nu$ is an unit outer normal vector on $\M$.
		Then the restriction  of the eigenfunction $e_\la(x)$ (denote also by $e_\la$ to simplify notations) to the boundary $\M$ is an eigenfunction of $\D+V$:
		\[(\D+V)e_\la(x)=\la e_\la(x),\ x\in \M.\] Here $\D$ is the Dirichlet-to-Neumann operator $\D$: $H^{\frac12}(\M)\to H^{-\frac12}(\M)$
		\[\D f=\partial_\nu u|_\M,\]
		where $u$ is the harmonic extension of $f$:
		\begin{equation}\label{harm}
		\begin{cases}\Delta_hu(x)=0,\ x\in \N, \\
			u(x)=f(x),\ x\in \partial \Omega=\M.
		\end{cases}\end{equation}
		
		Such a type of Steklov problem with potential has been considered in \cite{cox} from the point of view of conformal geometry, where the potential $V$ is the mean curvature on the boundary $\partial\Omega$. See e.g. \cite{esc1}, \cite{esc2}, \cite{esc3}, \cite{mar} for related works on Yamabe problem on compact manifolds with boundary. In the current paper, we derive estimates whenever the potential is merely bounded or Lipschitz.

		For $m\in \mathbb{R}$, we denote $OPS^m$  the class of pseudodifferential operators of order $m$.  It is known that $\D\in OPS^1$ and one can write (see e.g. \cite[Proposition C.1]{Taylor})
		\[\D=\sqrt{-\Delta_g}+P_0,\]
		for some $P_0\in OPS^0.$ 	Therefore, up to a classical pseudo-differential operator of order zero, the problem of eigenfunction bounds (among other results) on the boundary $M$ has been treated in our previous paper \cite{jmpa}. In this setting the model is related to relativistic matter (see e.g. \cite{carmona,daubechies,FLS1,FLS2,LiebYau1,LiebYau2}).

			In our first result below,  we provide a control of the $L^p$ norms of the Steklov eigenfunctions in the domain by their $L^2$ norms on the boundary.
		\begin{theorem}\label{thm2}Let $V\in L^\infty(M)$. Then for $\la\ge1$  we have
		\begin{equation}\label{intlp}	\|e_{\la}\|_{L^p(\Omega)}\ls \la^{-\frac1p+\sigma(p) }\|e_\la\|_{L^2(M)},\ \ 2\le p\le\infty,\end{equation}
	where\[
\sigma(p)=\begin{cases}\frac{n-1}2(\frac12-\frac1p),\ \ \  2\le p<\frac{2(n+1)}{n-1}\\
	\frac{n-1}2-\frac np,\ \ \ \ \ \  \frac{2(n+1)}{n-1}\le p\le \infty.\end{cases}\]		\end{theorem}
		The previous result is new, even for $V\equiv0$. Note that the estimate is sharp when $V\equiv 0$ and $\Omega$ is the unit ball $B(0,1)\subset \mathbb{R}^{n+1}$ with boundary $M=S^n$. In this case, the Steklov eigenfunction  $e_\la(x)=r^ke_k(\omega)$ in the polar coordinate $r\in [0,1]$, $\omega\in S^n$, where $\la^2=k(k+n-1)$, $k\in \mathbb{N}$ and $e_k(\omega)$ is a spherical harmonic of degree $k$, that is, the restriction to $S^n$ of homogeneous harmonic polynomials of degree $k$. It is straightforward to see that
		\begin{equation}\label{lplpsp}\|e_\la\|_{L^p(B(0,1))}\approx  \la^{-\frac1p}\|e_\la\|_{L^p(S^n)}.\end{equation} The following $L^p$ estimates of the Laplacian eigenfunctions  on the sphere $S^n$ are sharp
		\begin{equation}\label{sphere}\|e_\la\|_{L^p(S^n)}\ls \la^{\sigma(p)}\|e_\la\|_{L^2(S^n)},\end{equation}
		and they are saturated by zonal spherical harmonic for $p\ge \frac{2(n+1)}{n-1}$ and highest weight spherical harmonic for $p\le \frac{2(n+1)}{n-1}$ (see e.g. \cite{sogge08, sogge2015}). Thus, combining \eqref{lplpsp} with \eqref{sphere}, we see that \eqref{intlp} is sharp.

		The motivation for this result is to investigate the feature that Steklov eigenfunctions concentrate near the boundary, and rapidly decay away from the boundary (see e.g. \cite{HL01}, \cite{PST}, \cite{CP19}, \cite{GT19}). Motivated by the elliptic inverse boundary value problems such as Calderon problem (see e.g. \cite{cald}, \cite{man01}), Hislop-Lutzer \cite{HL01} proved that for any compact set $K\subset \Omega$, 
		\[\|e_\la\|_{L^2(K)}\le C_N \la^{-N}\|e_\la\|_{L^2(M)},\ \forall N.\]
		This bound reflects the fact that the Steklov eigenfunctions become highly oscillatory as the eigenvalue increases, hence they decay rapidly away from the boundary. Hislop-Lutzer \cite{HL01} conjectured that the decay is actually of order $e^{-\la d_g(K,\partial\Omega)}$. One may see by examining the case of unit ball $B(0,1)\subset \mathbb{R}^{n+1}$ that the exponential decay is optimal. This is confirmed
		 for real-analytic surfaces ($n=1$) by  Polterovich-Sher-Toth \cite{PST} and the eigenfunction decay is a key feature in their main results on nodal length. They proved that for any real-analytic compact Riemannian surface $\Omega$ with boundary $M=\partial\Omega$, and any compact set $K\subset \Omega$, there exist constants $C,c>0$ such that
		\[\|e_\la\|_{L^\infty(K)}\le C e^{-c\la d_g(K,\partial\Omega)}\|e_\la\|_{L^2(M)}.\] Their methods are specific to the case of real-analytic surfaces. A different method of proving this bound has been communicated to them by M. Taylor. Recently, Hislop-Lutzer's conjecture is  confirmed for higher dimensional real-analytic manifolds by Galkowski-Toth \cite{GT19}. Furthermore, this interesting concentration feature is also related to the restriction estimates of eigenfunctions to submanifolds (see e.g. \cite{bour09}, \cite{BGT07}, \cite{tohoku}, \cite{tataru})


		In our second result, we prove the following lower bound on the measure of the nodal set $$N_\la=\{x\in \M: e_\la(x)=0\}.$$
		
		\begin{theorem}\label{thm1}If $V\in Lip^1(\M)$ and zero is a regular value of $e_\la$, then
			\[	|N_\la|\gs \la^{\frac{3-n}2}.\]
		\end{theorem} 
		
		When $V\equiv 0$, this result is due to Wang-Zhu \cite{wangzhu}, which follows from the idea in Sogge-Zelditch \cite{sz2011}.  The Lipschitz assumption is used to ensure that the eigenfunctions is in $C^1$, so that the restriction of $\nabla e_\la$ to the nodal sets makes sense. The assumption that zero is a regular value is used to ensure the validity of  Gauss-Green theorem.

		To prove the theorems, we will need the following key lemmas. Incidentally, one does not require Lipschitz potentials but only bounded ones. 
		\begin{lemma}\label{lpnorm}If $V\in L^\infty(\M)$, then the following two eigenfunction estimates hold
  \begin{equation}\label{upp}
      \|e_\la\|_{L^p(\M)}\ls \la^{\sigma(p)}\|e_\la\|_{L^2(\M)},\,\,2\le p\le \infty.
  \end{equation}	
		Moreover,
			\begin{equation}\label{low} \|e_\la\|_{L^1(\M)}\gs \la^{-\frac{n-1}4}\|e_\la\|_{L^2(\M)}.\end{equation}
		\end{lemma}
	For smooth $V$, \eqref{upp} was proved by Seeger-Sogge \cite{ss}. Indeed, they obtained the eigenfunction estimates for self-adjoint elliptic  pesudo-differential operators satisfying a convexity assumtion on the principal symbol. In the case of the pure power (i.e. $P_0=0$), \eqref{upp} was stated in  \cite[Remark 1]{jmpa} by three of us. Both \eqref{upp} and \eqref{low} are sharp  on $S^n$. Indeed, they can be saturated by zonal spherical harmonic  or highest weight spherical harmonic (see e.g. \cite{sogge08, sogge2015}).
	
	\begin{lemma}\label{harmonic}If $V\in L^\infty(\M)$, then
		\begin{equation}\label{lplp}\|e_\la\|_{L^p(\Omega)}\ls \la^{-1/p}\|e_\la\|_{L^p(M)},\ 2\le p\le\infty.\end{equation}
	\end{lemma}
	The endpoint $p=\infty$ follows from the maximum principle, since $e_\la$ is harmonic in $\Omega$. The other endpoint $p=2$ can be obtained from the trace theorem and standard regularity estimates. And then \eqref{lplp} is proved by an interpolation argument involving the harmonic extension operator on $\Omega$. From \eqref{lplpsp}, we see that the estimate \eqref{lplp} is sharp for $\Omega=B(0,1)$.

	The paper is organized as follows. In Section 2, we prove sharp heat kernel estimates that will be used later. In Section 3, we prove Lemma \ref{lpnorm}. In section 4, we prove some kernel estimates for pseudo-differential operators on compact manifolds. In Section 5, we prove the interior eigenfunction estimates in Theorem \ref{thm2}. In Section 6, we prove the size estimates of the nodal sets in Theorem \ref{thm1}.

		Throughout this paper, $X\ls Y$ (or $X\gs Y$) means $X\le CY$ (or $X\ge C Y$) for some positive constant $C$ independent of $\la$. This constant may depend on $V$ and the domain $\Omega$. $X\approx Y$ means $X\ls Y$ and $X\gs Y$. 
	\section{Heat kernel bounds}
	
	In this section, we prove the  heat kernel estimates for the operators $$H_V= (-\lap)^{\alpha/2}+P_{\alpha-1}+V,$$ where $P_{\alpha-1}$ is a classical pseudo-differential operator of order $\alpha-1$, and the real-valued potential $V$ belongs to the Kato class on the closed manifold $(M,g)$. These generalize the results of Gimperlein-Grubb \cite[Theorem 4.3]{GG}.  When $P_{\alpha-1}=0$, the Euclidean version was proved in \cite{cks} and \cite{song}.  We  give a detailed proof  for this special case on compact manifolds by using Duhamel's principle and Picard iterations.  And then we slightly modify this argument to obtain the upper bound of the heat kernel of $H_V$. Although the potentials in our main theorems are just bounded, we prove the heat kernel estimates under the minimal assumption so that they may be used for related reseach.

	\begin{definition}\label{fractionalkato}
		For $n\ge2$ and $0<\alpha<2$, the potential $V$ is said to be in the Kato class $\mathcal{K}_\alpha(M)$ if
		\begin{equation}\label{kato}
			\lim_{r\downarrow 0}\sup_{x\in M}\int_{B_r(x)}d_g(x,y)^{\alpha-n}|V(y)|dy=0
		\end{equation}
		where $d_g(\cdot,\cdot)$ denotes geodesic distance and $B_r(x)$ is the geodesic ball of radius $r$ about $x$ and $dy$ denotes the volume element on $(M,g)$. To define the Kato class for $n=1$ and $0<\alpha<2$, we replace the function $d_g(x,y)^{\alpha-n}$ in \eqref{kato} by 
		\[w(x,y)=\begin{cases}
			d_g(x,y)^{\alpha-1},\ \ \quad\quad\quad\ \alpha<1\\
			\log(2+d_g(x,y)^{-1}),\ \ \ \alpha=1\\
			1,\ \ \quad\quad\quad\quad\quad\quad\quad\quad \alpha>1.
		\end{cases}\]
	\end{definition}
	 Since $M$ is compact we have $ \mathcal{K}_\alpha(M)\subset L^1(M)$, and  for any $p>\frac{n}\alpha$, we have $L^{p}(M)\subset\mathcal{K}_\alpha(M)$ by H\"older's inequality. We recall that the assumption $V\in \mathcal{K}_\alpha(M)$ implies that the operators $H_V=(-\Delta_g)^{\alpha/2}+V$ are self-adjoint and bounded from below. See the proof of \cite[Proposition 2]{jmpa}. The same argument is still valid to prove that $H_V=(-\Delta_g)^{\alpha/2}+P_{\alpha-1}+V$ is  self-adjoint  and bounded from below, whenever $P_{\alpha-1}$ is self-adjoint. 
	
	\begin{proposition}\label{heatbd}
		Let $n\ge1$, $0<\alpha<2$ and $t>0$. Let $p^V(t,x,y)$ be the heat kernel of $H_V=(-\Delta_g)^{\alpha/2}+V$, where $V\in \mathcal{K}_\alpha(M)$. Then for any $t\in(0,1]$, $x,y \in M$
		\begin{equation}\label{heat1}
			p^V(t,x,y)\approx q_\alpha(t,x,y)
		\end{equation}
		where $q_\alpha(t,x,y)=\min\{t^{-n/\alpha},\  td_g(x,y)^{-n-\alpha}\}$. Moreover, for any $t>0$, $x,y \in M$
		\begin{equation}\label{heat2}
			e^{-C_1t}q_\alpha(t,x,y)\ls p^V(t,x,y)\ls e^{C_2t}q_\alpha(t,x,y)
		\end{equation}
		for some constants $C_1,C_2>0$.
	\end{proposition}
	\begin{proposition}\label{pdoheat}
		Let $n\ge1$, $0<\alpha<2$ and $t>0$. Let $p^V(t,x,y)$ be the heat kernel of $H_V=(-\Delta_g)^{\alpha/2}+P_{\alpha-1}+V$, where $V\in \mathcal{K}_\alpha(M)$ and $P_{\alpha-1}$ is a classical pseudo-differential operator of order $\alpha-1$. Then for any $t>0$, $x,y \in M$
		\begin{equation}\label{heatpdo}
			|p^V(t,x,y)|\ls e^{Ct} q_\alpha(t,x,y).
		\end{equation}
		for some constant $C>0$.
	\end{proposition}
	The following key lemma is called 3P-inequality in \cite[Theorem 4]{bj07} and \cite[Proposition 2.4]{wz15}. We remark that such 3P-inequality holds for all $\alpha\in(0, 2)$ but fails to hold for the Gaussian kernel ($\alpha=2$).
	\begin{lemma}\label{3p}We have for any $s,t>0$ and $x,y,z\in M$
		\[q_\alpha(t,x,z)q_\alpha(s,z,y)\le Cq_\alpha(s+t,x,y)(q_\alpha(t,x,z)+q_\alpha(s,z,y)),\]
		where $C>0$ is a constant.
	\end{lemma}
	\begin{proof}
	The proof is straightforward. Indeed, by using the fact that for $A,B>0$ \[\min\{A,B\}\approx \frac{AB}{A+B},\]
	\[(A+B)^{\frac n\alpha}\approx A^{\frac n\alpha}+B^{\frac n\alpha},\]
	and the triangle inequality $d_g(x,y)\le d_g(x,z)+d_g(z,y)$, we have
	\begin{align*}
		\frac{q_\alpha(t,x,z)+q_\alpha(s,z,y)}{q_\alpha(t,x,z)q_\alpha(s,z,y)}&=\frac1 {q_\alpha(t,x,z)}+\frac1{q_\alpha(s,z,y)}\\
		& \approx t^{\frac n\alpha}+t^{-1}d_g(x,z)^{n+\alpha}+s^{\frac n\alpha}+s^{-1}d_g(z,y)^{n+\alpha}\\
		&\approx (t+s)^{\frac n\alpha}+t^{-1}d_g(x,z)^{n+\alpha}+s^{-1}d_g(z,y)^{n+\alpha}\\
		&\ge (t+s)^{\frac n\alpha}+(s+t)^{-1}(d_g(x,z)^{n+\alpha}+d_g(z,y)^{n+\alpha})\\
		&\approx (t+s)^{\frac n\alpha}+(s+t)^{-1}(d_g(x,z)+d_g(z,y))^{n+\alpha}\\
		&\ge (t+s)^{\frac n\alpha}+(s+t)^{-1}d_g(x,y)^{n+\alpha}\\
		&\approx \frac1 {q_\alpha(s+t,x,y)}.
	\end{align*}
	The implicit constants may depend on $n$ and $\alpha$. This completes the proof of Lemma \ref{3p}. 
	\end{proof}

	\noindent \textbf{Proof of Proposition \ref{heatbd}.} It is not hard to see that \eqref{heat2} follows from \eqref{heat1} and the semigroup property. So it suffices to prove \eqref{heat1}.
	
	Since $(M,g)$ is a closed manifold,  the heat kernel of $-\Delta_g$ satisfies the two-sided estimates (see Li-Yau \cite{LY}, Sturm \cite{sturm}, Saloff-Coste \cite{sal})
	\[t^{-n/2}e^{-C_1d_g(x,y)^2/t}\ls p_t(x,y)\ls t^{-n/2}e^{-C_2d_g(x,y)^2/t},\ t>0,\ x,y\in M\]
	for some constants $C_1,C_2>0$. Moreover, it is well-known that the semigroups $e^{t\Delta_g}$ and $e^{-t(-\Delta_g)^{\alpha/2}}$ are related by subordination formulas (see e.g. \cite[(4.8)]{GG}, \cite{gri03}, \cite{zo86}), which imply that the heat kernel of $H^0=(-\Delta_g)^{\alpha/2}$ is continuous and satisfies the two-sided estimates (see e.g. \cite[Theorem 4.2]{GG}, \cite[Theorem 3.1]{BSS})
	\begin{equation}\label{p0}C^{-1} q_\alpha(t,x,y)\le p_0(t,x,y)\le C q_\alpha(t,x,y),\ t>0,\ x,y\in M.\end{equation}
	The heat kernel	$p_0(t,x,y)$ is the Schwartz kernel of $f\to e^{-tH^0}f=u^0(t,x)$, which solves the heat equation
	\begin{equation}\label{heat0}\begin{cases}(\partial_t+ H^0)u^0(t,x)=0,\ \ (t,x)\in (0,\infty)\times M,\\
			u^0|_{t=0}=f.\end{cases}\end{equation}
	Similarly, the heat kernel	$p^V(t,x,y)$ is the Schwartz kernel of $f\to e^{-tH_V}f=u_V(t,x)$, which solves the heat equation
	\begin{equation}\label{heatv}\begin{cases}(\partial_t+ H_V)u_V(t,x)=0,\ \ (t,x)\in (0,\infty)\times M,\\
			u_V|_{t=0}=f.\end{cases}\end{equation}
	Note that \eqref{heat0} and \eqref{heatv} imply that
	\[(\partial_t+H^0)(e^{-tH_V}f-e^{-tH^0}f)=-V(x)e^{-tH_V}f\]
	and 
	\[(e^{-tH_V}f-e^{-tH^0}f)|_{t=0}=0.\]
	By Duhamel’s principle for the heat equation, we have
	\begin{align*}
		e^{-tH_V}f-e^{-tH^0}f&=-\int_0^te^{-(t-r)H^0}(Ve^{-r H_V}f)dr\\
		&=-\int_0^t\int_M\int_M p_0(t-r,x,z)V(z)p^V(r,z,y)f(y)dydzdr.
	\end{align*}
	where $dy$ and $dz$ denote the volume element on $(M,g)$.
	So the heat kernel of $H_V$ satisfies the integral equation
	\begin{equation}\label{inteq}
		p^V(t,x,y)=p_0(t,x,y)-\int_0^t \int_M p_0(t-r,x,z)p^V(r,z,y)V(z)dzdr.
	\end{equation}
	To prove \eqref{heat1}, we use Picard iterations (see e.g.  \cite{bj07}, \cite{wz15}) to construct a solution to \eqref{inteq}.
	For $t>0$, $x,y\in M$, let \begin{equation}\label{pn}
		p_m(t,x,y)=p_0(t,x,y)-\int_0^t \int_M p_0(t-r,x,z)p_{m-1}(r,z,y)V(z)dzdr,\ \ m\ge1.
	\end{equation}
	Moreover, let
	\[\Theta_m(t,x,y)=p_m(t,x,y)-p_{m-1}(t,x,y),\ \ m\ge1\]
	and $\Theta_0(t,x,y)=p_0(t,x,y)$. Clearly,
	\begin{equation}\label{induct}\Theta_m(t,x,y)=-\int_0^t\int_M p_0(t-r,x,z)V(z) \Theta_{m-1}(r,z,y)dzdr.\end{equation}
	We claim that for some constant $c_0>0$ and $\ c(t)>0$
	\begin{equation}\label{claim}|\Theta_m(t,x,y)|\le (c_0c(t))^mp_0(t,x,y),\ \ m\ge0.\end{equation}
  To prove the claim, we define
	\begin{equation}\label{ct}
		c(t)=\sup _{y\in M}  \int_0^t \int_M q_\alpha(r,y,z)|V(z)|dzdr.
	\end{equation} It is straightforward to see that $V$ is in the Kato class implies that 
	\begin{equation}\label{clim}\lim_{t\downarrow0}c(t)=0.\end{equation}  Indeed, for $n\ge2$,
	\begin{align*}
		\int_0^t \int_M q_\alpha(r,y,z)|V(z)|dzdr\ls \int_{d_g(z,y)<t^{\frac1{2\alpha}}}d_g(z,y)^{\alpha-n}|V(z)|dz+\int_M td_g(z,y)^{\alpha-n}|V(z)|dz,
		\end{align*}
	which implies \eqref{clim} by the definition \eqref{kato}. The case  $n=1$ is similar.

The claim \eqref{claim} is clear for $m=0$.	If the claim  is true for $m-1$, then by \eqref{induct} we have
	\begin{align*}|\Theta_{m}(t,x,y)|&\le (c_0c(t))^{m-1}\int_0^t \int_M p_0(t-r,x,z)p_0(r,z,y)|V(z)|dzdr\\
		&\le C(c_0c(t))^{m-1}\int_0^t \int_M p_0(t,x,y)(p_0(t-r,x,z)+p_0(r,z,y))|V(z)|dzdr\\
		&\le 2C^2(c_0c(t))^{m-1}p_0(t,x,y)\sup_{y\in M}\int_0^t \int_M q_\alpha(r,y,z)|V(z)|dzdr\\
		&= \frac{2C^2}{c_0}(c_0c(t))^{m}p_0(t,x,y)
	\end{align*}
	where we use Lemma \ref{3p} and the upper bound in \eqref{p0}. 
	Here $C>0$ is a constant independent of $m, s, t,x,y,z$. So we may fix $c_0\ge 2C^2$, and the claim \eqref{claim} is proved by induction.

	By \eqref{clim}, there is $0<t_0<1$ so that for any $t\in(0,t_0]$, we have $c_0c(t)\le\frac13$. Let
	\[p^V(t,x,y)=\sum_{m=0}^\infty \Theta_m(t,x,y),\ \ t\in(0,t_0],\ x,y\in M.\]
	This series is uniformly convergent, and
	\[|p^V(t,x,y)-p_0(t,x,y)|\le  \sum_{m=1}^\infty|\Theta_m(t,x,y)|\le \frac{c_0c(t)}{1-c_0c(t)}p_0(t,x,y)\le \frac12 p_0(t,x,y).\]
	Combining this with \eqref{p0}, we have
	\begin{equation}\label{estp}
		p^V(t,x,y)\approx q_\alpha(t,x,y),\ \ t\in(0,t_0].
	\end{equation}
	By letting $m\to \infty$ in \eqref{pn}, we get \eqref{inteq} for $t\in(0,t_0]$.
	
	Moreover, when $t\in(0,t_0]$, $p^V(t,x,y)$ is the unique solution to the integral equation \eqref{inteq} satisfying \eqref{estp}. Indeed, let $\tilde p(t,x,y)$ be another solution satisfying \eqref{estp}, and $\Theta=p^V-\tilde p$. Note that $|\Theta(t,x,y)|\le C p_0(t,x,y)$ for some constant $C>0$. Then by the same induction argument above  we obtain
	\[|\Theta(t,x,y)|\le C(c_0c(t))^mp_0(t,x,y),\ \forall m\ge0.\]
	By letting $m\to\infty$ we get $\Theta(t,x,y)=0$ for $t\in (0,t_0]$.

	For $t>t_0$, we recursively define \[p^V(t,x,y)=\int_M p^V(t/2,x,z)p^V(t/2,z,y)dz.\] Then $p^V(t,x,y)$ is extended to be a jointly continuous  function on $(0,\infty)\times M\times M$. Moreover, the estimate \eqref{estp} can be recursively extended to
	\[		p^V(t,x,y)\approx q_\alpha(t,x,y),\ \ t\in(0,1].\]
	This completes the proof of Proposition \ref{heatbd}.

	\noindent \textbf{Proof of Proposition \ref{pdoheat}.} The proof is similar to Proposition \ref{heatbd}. It suffices to prove \eqref{heatpdo} for $t\in (0,1]$ by the semigroup property. Then the argument above is still valid for \eqref{heatpdo}, if we  replace \eqref{p0} by the heat kernel  bounds of $H^0=(-\Delta_g)^{\alpha/2}+P_{\alpha-1}$ (see \cite[Theorem 4.3]{GG})
	\[|p_0(t,x,y)|\le C q_{\alpha}(t,x,y),\ \ t\in(0,1],\ x,y\in M.\]

		\section{Global eigenfunction estimates: proof of Lemma \ref{lpnorm}}

To prove Lemma~\ref{lpnorm}, we begin with the following resolvent estimate.
\begin{proposition}\label{resolvinfty}
   For $\la\ge1$, we have
\begin{equation}\label{resolvv}
    \|(\sqrt{-\Delta_g}-(\la+i))^{-1}\|_{L^2\to L^p}\ls \la^{\sigma(p)},\ 2<p\le \tfrac{2(n+1)}{n-1},
\end{equation}
where \begin{equation}\label{soggeexponent}
    \sigma(p)=\begin{cases}\frac{n-1}2(\frac12-\frac1p),\ \ \  2\le p<\frac{2(n+1)}{n-1}\\
		\frac{n-1}2-\frac np,\ \ \ \ \ \  \frac{2(n+1)}{n-1}\le p\le \infty.\end{cases}
\end{equation}
\end{proposition}
\begin{proof}
For $k\in \mathbb{N}$, let $\chi_{[k,k+1)}$ denote the spectral projection operators for $\sqrt{-\Delta_g}$ corresponds to the spectral interval $[k,k+1)$ , and let $\chi_{[2[\la],\infty)}$ spectral projection operator onto the interval $[2[\la],\infty)$, where $[\la]$ denotes the largest integer that is smaller than $\la$.
 Then for any function $f$, by Cauchy-Schwarz inequality
 \begin{equation}
     \begin{aligned}
         (\sqrt{-\Delta_g}-(\la+i))^{-1}f
         &=\sum_{k<2[\la]} \frac{1}{k-(\la+i)} (k-(\la+i))(\sqrt{-\Delta_g}-(\la+i))^{-1}\chi_{[k,k+1)}f \\
         &\quad+\chi_{[2[\la],\infty)}(\sqrt{-\Delta_g}-(\la+i))^{-1}f\\
         &\ls (\sum_{k<2[\la]} \big|(k-(\la+i))(\sqrt{-\Delta_g}-(\la+i))^{-1}\chi_{[k,k+1)}f\big|^2)^{\frac12}\\
          &\quad+\big|\chi_{[2[\la],\infty)} (\sqrt{-\Delta_g}-(\la+i))^{-1}f\big|
     \end{aligned}
 \end{equation}
Thus, by Minkowski's inequality
 \begin{equation}
     \begin{aligned}
        \| (\sqrt{-\Delta_g}-(\la+i))^{-1}f \|_{L^p}&\le (\sum_{k<2[\la]} \|(k-(\la+i))(\sqrt{-\Delta_g}-(\la+i))^{-1}\chi_{[k,k+1)}f\|_{L^p}^2)^{\frac12} \\
        &\quad+\|\chi_{[2[\la],\infty)} (\sqrt{-\Delta_g}-(\la+i))^{-1}f\|_{L^p}
     \end{aligned}
 \end{equation}
 
 To handle the first term on the right, 
note that $\chi_{[k,k+1)}=\chi_{[k,k+1)}\circ \chi_{[k,k+1)}$, and by the classical results in \cite{sogge88}, 
\begin{equation}
    \|\chi_{[k,k+1)}f\|_{L^p}\ls (1+k)^{\sigma(p)} \|f\|_{L^2}\ls \la^{\sigma(p)} \|f\|_{L^2},\,\,\,\text{if}\,\,k<2[\la].
\end{equation}
Thus, 
\begin{equation}
     \begin{aligned}
  \sum_{k<2[\la]}& \|(k-(\la+i))(\sqrt{-\Delta_g}-(\la+i))^{-1}\chi_{[k,k+1)}f\|_{L^p}^2)^{\frac12} \\
  &\ls \la^{\sigma(p)}  \sum_{k<2[\la]} \|(k-(\la+i))(\sqrt{-\Delta_g}-(\la+i))^{-1}\chi_{[k,k+1)}f\|_{L^2}^2)^{\frac12} \\
    &\ls \la^{\sigma(p)}  \sum_{k<2[\la]} \|\chi_{[k,k+1)}f\|_{L^2}^2)^{\frac12}\\
        &\ls \la^{\sigma(p)} \|f\|_{L^2},
     \end{aligned}
 \end{equation}
where in the second inequality we used the fact that by spectral theorem,
\begin{equation}
    \|(k-(\la+i))(\sqrt{-\Delta_g}-(\la+i))^{-1}\chi_{[k,k+1)}f\|_{L^2}\ls  \|\chi_{[k,k+1)}f\|_{L^2}, \,\,\forall\, k\in \mathbb{N}.
\end{equation}

To handle the second term, we use Sobolev estimates to see that
\begin{equation}
\begin{aligned}
    \|&\chi_{[2[\la],\infty)} (\sqrt{-\Delta_g}-(\la+i))^{-1}f\|_{L^p}\\
    &\ls  \|\chi_{[2[\la],\infty)}(\sqrt{-\Delta_g})^{n(\frac12-\frac{1}{p})} (\sqrt{-\Delta_g}-(\la+i))^{-1}f\|_{L^2}.
\end{aligned}
\end{equation}
 When $2<p\le \tfrac{2(n+1)}{n-1}$, it is straightforward to check that $n(\frac12-\frac{1}{p})<1$, thus by spectral theorem,
 \begin{equation}
    \|\chi_{[2[\la],\infty)}(\sqrt{-\Delta_g})^{n(\frac12-\frac{1}{p})} (\sqrt{-\Delta_g}-(\la+i))^{-1}f\|_{L^2}\ls \|f\|_{L^2},
\end{equation}
which is better than the desired bound in \eqref{resolvv}.
 \end{proof}
 
Now we shall prove Lemma~\ref{lpnorm}, this follows from  similar strategies as in \cite{blair2021uniform}.
Recall that $\D=\sqrt{-\Delta_g}+P_0$, by using the second resolvent formula, we have 
\begin{multline}\label{dvresolvent}
    (\D+V-(\la+i))^{-1}=(\sqrt{-\Delta_g}-(\la+i))^{-1} \\-(\sqrt{-\Delta_g}-(\la+i))^{-1}(P_0+V)(\D+V-(\la+i))^{-1}.
\end{multline}
		Since $P_0\in OPS^0$ and the eigenvalues of $\D+V$ are real, by spectral theorem, we have
  \begin{equation}\label{p0dv}
      \|P_0 (\D+V-(\la+i))^{-1}\|_{L^2\to L^2}\ls \|(\D+V-(\la+i))^{-1}\|_{L^2\to L^2}\ls 1.
  \end{equation}
Similarly, since $ V\in L^\infty(M)$, we have 
  \begin{equation}\label{vlinfty}
      \|V (\D+V-(\la+i))^{-1}\|_{L^2\to L^2}\ls 1.
  \end{equation}
Thus, \eqref{dvresolvent}, \eqref{p0dv}, \eqref{vlinfty} and \eqref{resolvv} yield that 
\begin{equation}\label{quasimode resolvent}
    \|(\D+V-(\la+i))^{-1}\|_{L^2\to L^p}\ls \la^{\sigma(p)},\ 2<p\le \tfrac{2(n+1)}{n-1}.
\end{equation}

If we let $\chi^V_{[\la,\la+1)}$ denote the spectral projection operator associated with $\sqrt{-\Delta_g}+P_0+V$ for the interval $[\la,\la+1)$, then \eqref{quasimode resolvent}  implies the following
\begin{corollary}\label{spectralv}
     Let $V\in L^\infty(M)$, we have
     \begin{equation}\label{spectralv1}
         \|\chi^V_{[\la,\la+1)}f\|_{L^p}\ls \la^{\sigma(p)}\|f\|_{L^2},\ 2<p\le \infty.
     \end{equation}
\end{corollary}
Note that if we take $f=e_\la$ in \eqref{spectralv1}, and use the fact that $\chi^V_{[\la,\la+1)} e_\la=e_\la$, we obtain \eqref{upp}. 
\begin{proof}[Proof of Corollary~\ref{spectralv}]
  If $2<p\le \tfrac{2(n+1)}{n-1}$, this follows from \eqref{quasimode resolvent} by letting $f=\chi^V_{[\la,\la+1)}f$ there along with the fact that 
\begin{equation}\label{quasimode resolvent1}
    \|(\D+V-(\la+i))\chi^V_{[\la,\la+1)}\|_{L^2\to L^2}\ls 1.
\end{equation}

  If $p>\tfrac{2(n+1)}{n-1}$, we shall use the heat kernel bounds in Proposition \ref{pdoheat}. More explicitly, let $H_V=\sqrt{-\Delta_g}+P_0+V$, note that if $V\in L^\infty(M)$, then \eqref{kato} holds with $\alpha=1$, thus $V\in \mathcal{K}_1(M)$, which, by Proposition \ref{pdoheat},  implies that we have the kernel estimate \eqref{heatpdo} for $e^{-tH_V}$.  As a result, 
by \eqref{heatpdo} and Young's inequality, we have the following:
\begin{equation}\label{young}
 \|e^{-tH_V}\|_{L^p(M)\rightarrow L^q(M)} \ls t^{- n(\frac 1p-\frac 1q)}, \quad \text{if}\hspace{2mm} 0<t\le1, \ \text{and} \hspace{2mm} 1\le p\le q\le \infty.
\end{equation}  
If we fix $t=\la^{-1}$ and $p_c=\frac{2(n+1)}{n-1}$, and apply the above bound, we have for $p>\frac{2(n+1)}{n-1}$, 
  \begin{equation}\label{soblev chiv}
  \begin{aligned}
          \|\chi^V_{[\la,\la+1)}f\|_{L^p}&\ls  \la^{n(\frac 1{p_c}-\frac 1p)}\|e^{\la^{-1}H_V}\chi^V_{[\la,\la+1)}f\|_{L^{p_c}} \\
          &= \la^{n(\frac 1{p_c}-\frac 1p)}\|\chi^V_{[\la,\la+1)}e^{\la^{-1}H_V}\chi^V_{[\la,\la+1)}f\|_{L^{p_c}}\\
          & \ls \la^{n(\frac 1{p_c}-\frac 1p)}\la^{\frac{n-1}2-\frac n{p_c}}\|e^{\la^{-1}H_V}\chi^V_{[\la,\la+1)}f\|_{L^{2}} \\
          & \ls \la^{\frac{n-1}2-\frac n{p}}\|f\|_{L^{2}},
  \end{aligned}
     \end{equation}
where in the third line we applied \eqref{spectralv1} at $p=p_c$ and in the last line we  applied spectral theorem. Since $\tfrac{n-1}2-\tfrac n{p}=\sigma(p)$ when 
 $p\ge p_c$, the proof of Corollary~\ref{spectralv} is complete.
\end{proof}
To prove Lemma~\ref{lpnorm} it remains to prove \eqref{low}. By using the arguments from Sogge-Zelditch \cite{sz2011}, we note that  \eqref{low} can be obtained from H\"older's inequality and \eqref{upp}
		\[\|e_\la\|_{L^2(\M)}^{\frac1\theta}\le\|e_\la\|_{L^1(\M)} \|e_\la\|_{L^p(\M)}^{\frac1\theta-1}\ls \|e_\la\|_{L^1(\M)} (\la^{\sigma(p)}\|e_\la\|_{L^2(\M)})^{\frac1\theta-1}=\|e_\la\|_{L^1(\M)}\la^{\frac{n-1}4}\|e_\la\|_{L^2(\M)}^{\frac1\theta-1}.\]
		Here $2<p<\frac{2(n+1)}{n-1}$, and $\theta=\frac{p}{p-1}(\frac12-\frac1p)$.

\section{Kernels of Pseudo-differential operators}
 In this section, we prove a useful lemma concerning the kernel estimates of the pseudo-differential operators on compact manfolds. 
 \begin{lemma}\label{pdo}
  	Let $\mu\in \mathbb{R}$, and $m\in C^\infty(\mathbb{R})$ belong to the symbol class $S^\mu$, that is, assume that
\begin{equation}\label{symb}|\partial_t^{\alpha} m(t)| \leq C_{\alpha}(1+|t|)^{\mu-\alpha}, \quad \forall \alpha.  	\end{equation}
  	If $P=\sqrt{-\Delta_g}$, then $m(P)$ is a pseudo-differential operator of order $\mu$. Moreover, if $R\ge1$, then the kernel of the operator $m(P/R)$ satisfies for all $N\in\mathbb{N}$
  	\begin{equation}\label{withla}
  		|m(P/R)(x,y)|\\\ls  \begin{cases}R^n\big(R d_g(x,y)\big)^{-n-\mu}\big(1+R d_g(x,y)\big )^{-N},\,\quad\quad\,\,\,\, \ n+\mu>0\\
  			R^n\log(2+(Rd_g(x,y))^{-1})\big(1+R d_g(x,y)\big)^{-N},\,\, \ n+\mu=0\\
  			R^n(1+Rd_g(x,y))^{-N},\ \ \ \ \ \quad\quad\quad\quad\quad\quad\quad\quad\ \  n+\mu<0.
  		\end{cases}
  	\end{equation}
  \end{lemma}
  See \cite[Theorem 4.3.1]{fio} for the proof of the fact that $m(P)$ is a pseudo-differential operator of order $\mu$.   The kernel bounds \eqref{withla} can be viewed as the rescaled version on compact manifolds compared to the Euclidean estimates in \cite[Proposition 1 on page 241]{steinbook}. We mean that the bounds hold near the diagonal (so that $d_g(x,y)$ is smaller than the injectivity radius of $M$) and that   outside the neighborhood of the diagonal  they are $O(R^{-N})$.  Roughly speaking, modulo lower order terms, $m(P/R)(x, y)$  equals
  \[(2\pi)^{-n}\int_{\mathbb{R}^n}m(|\xi|/R)e^{i d_g(x,y)\xi_1}d\xi\]
  near the diagonal, which satisfies the bounds in \eqref{withla}, while outside of a fixed neighborhood of the diagonal  $m(P/R)(x, y)$ is $O(R^{-N})$. For completeness, we give a detailed proof by using the Hadamard parametrix.

\noindent \textbf{Proof of  \eqref{withla}.} Since the spectrum of $P=\sqrt{-\Delta_g}$ is nonnegative, we may assume that $m(t)$ is an even function on $\mathbb{R}$. Let $\delta>0$ be smaller than the injectivity radius of $(M,g)$. Let $\rho\in C_0^\infty(-1,1)$ be even and satisfy $\rho\equiv1$ on $(-\frac\delta2,\frac\delta2)$. So we can write
\begin{align}\label{intsum}m(P/R)&=\frac R{2\pi}\int_{\mathbb{R}} \hat m(tR) \cos (tP)dt \nonumber\\ 
	&=\frac R{2\pi}\int \rho(t)\hat m(tR)\cos (tP)dt+\frac R{2\pi}\int (1-\rho(t))\hat m(tR)\cos (tP)dt.
	\end{align}
  To handle the first term in \eqref{intsum}, we need to use the Hadamard parametrix (see e.g. \cite[Section 1.2 and Theorem 3.1.5]{hangzhou}).  For $0<t<\delta$ and $N_0>n+3$, we have 
 \begin{align}\label{costp}
  	\cos tP(x,y)=\sum_{\nu=0}^{N_0}\omega_\nu(x,y)\partial_t E_\nu(t,d_g(x,y))+R_{N_0}(t,x,y)  \end{align} 
 where the leading term \begin{equation}\label{E0}\partial_tE_0=(2\pi)^{-n}\int_{\mathbb{R}^n} e^{id_g(x,y)\xi_1}\cos (t|\xi|)d\xi\end{equation}
 and $E_\nu$ satisfies $2\partial_t E_\nu=t E_{\nu-1}$, and $\partial_t E_\nu(t/R,r)=R^{n-2\nu}\partial_t E_\nu(t,Rr)$ for any $R>0$. Here $\omega_\nu \in C^\infty(M\times M)$, and $\omega_0(x,x)=1,\forall x \in M$. For $\nu\ge1$, we have the following explicit formula (see e.g. \cite[Section 1.2]{hangzhou})
 \begin{align*}
 	E_\nu=\nu!(2\pi)^{-n}\int_{0\le s_1\le...\le s_\nu\le t}\int_{\mathbb{R}^n} e^{id_g(x,y)\xi_1}\frac{\sin(t-s_\nu)|\xi|}{|\xi|}&\frac{\sin(s_\nu-s_{\nu-1})|\xi|}{|\xi|}\cdot\cdot\cdot\\
 	&\cdot\frac{\sin(s_2-s_{1})|\xi|}{|\xi|}\frac{\sin s_1|\xi|}{|\xi|}d\xi ds_1...ds_\nu.
 \end{align*}
 So for $\nu\ge1$ we can  obtain (see e.g. \cite[Section 1.2]{hangzhou})
 	\begin{align}\label{Ev}\partial_tE_\nu=\tfrac12 tE_{\nu-1}&=\int e^{id_g(x,y)\xi_1} a_\nu(t,|\xi|)d\xi\\
 	&=\sum_{\pm}\sum_{j=0}^{\nu-1}a_{j\nu}^\pm\int e^{id_g(x,y)\xi_1\pm it|\xi|}t^{j+1}|\xi|^{-2\nu+1+j}d\xi,\nonumber\end{align}
 where  $a_{j\nu}^\pm$ are constants, and $a_\nu\in C^\infty$.
The remainder kernel $R_{N_0}\in C^{N_0-n-3}$ satisfies
  \begin{equation}\label{err}
  	|\partial^\alpha_{t,x,y}R_{N_0}(t,x,y)|\ls |t|^{2N_0+2-n-|\alpha|},\ \ |\alpha|\le N_0-n-2.
  \end{equation}

Then we plug  \eqref{costp} into the first term of \eqref{intsum}. We first handle the contribution of the leading term in \eqref{costp}. By \eqref{E0}, we can write 
  \begin{align*}
  \frac R{2\pi}\iint\rho(t)\hat m(tR)&\cos (t|\xi|)e^{id_g(x,y)\xi_1}dtd\xi= \int m(|\xi|/R)e^{id_g(x,y)\xi_1}d\xi+\\
  & \ \ \ \ \frac R{2\pi}\iint (1-\rho(t))\hat m(tR)\cos (t|\xi|)e^{id_g(x,y)\xi_1}dtd\xi:=I_1+I_2.
    \end{align*}
Using the property \eqref{symb} and integration by parts, we see that for any $N\in\mathbb{N}$
 \begin{align}\label{I1}
 |I_1|\ls \begin{cases}R^n\big(R d_g(x,y)\big)^{-n-\mu}\big(1+R d_g(x,y)\big )^{-N},\,\quad\quad\,\,\,\, \ n+\mu>0\\
 		R^n\log(2+(Rd_g(x,y))^{-1})\big(1+R d_g(x,y)\big)^{-N},\,\, \ n+\mu=0\\
 		R^n(1+Rd_g(x,y))^{-N},\ \ \ \ \ \quad\quad\quad\quad\quad\quad\quad\quad\ \  \ n+\mu<0
 	\end{cases}
\end{align}
and
  \begin{align}\label{I2}
|I_2|&\ls \Big| R\iiint (1-\rho(t))(tR)^{-N} m^{(N)}(s)e^{-itRs}\cos (t|\xi|)e^{id_g(x,y)\xi_1}dsdtd\xi\Big|\nonumber\\
  &\ls R^{-N+1}\iint (1+||\xi|-R|s||)^{-N_1}(1+|s|)^{-N+\mu}dsd\xi\nonumber\\
  	&\ls R^{-N}\int (1+|\xi|/R)^{-N+\mu}d\xi\nonumber\\
  &\ls R^{-N+n}.
  \end{align}
Here we choose $N_1>N>n+\mu$. 

 Similarly, we can handle the contributions of the remaining terms in \eqref{costp}. For each $\nu\ge1$, we can write 
  \begin{align*}
	\frac R{2\pi}\int \rho(t)\hat m(tR)\partial_t &E_\nu (t,d_g(x,y))dt
	= \frac R{2\pi}\int\hat m(tR)\partial_t E_\nu (t,d_g(x,y))dt-\\
	&  \ \ \ \ \frac R{2\pi}\int (1-\rho(t))\hat m(tR)\partial_t E_\nu (t,d_g(x,y))dt:=I_3+I_4.
\end{align*}
Using the scaling property $\partial_t E_\nu(t/R,r)=R^{n-2\nu}\partial_t E_\nu(t,Rr)$ and the formula \eqref{Ev}, we can integrate by parts to see that
 \begin{align}|I_3|&=(2\pi)^{-1} R^{n-2\nu}\Big|\int \hat m(t)\partial_t E_\nu(t, Rd_g(x,y))dt\Big|\nonumber\\
 &= (2\pi)^{-1}R^{n-2\nu}\Big|\sum_{\pm}\sum_{j=0}^{\nu-1}a_{j\nu}^\pm\iint e^{iRd_g(x,y)\xi_1\pm it|\xi|}\hat m(t)t^{j+1}|\xi|^{-2\nu+1+j}dtd\xi \Big|\nonumber\\
 &= R^{n-2\nu}\Big|\sum_{\pm}\sum_{j=0}^{\nu-1}i^{-j-1}a_{j\nu}^\pm \int e^{iRd_g(x,y)\xi_1}m^{(j+1)}(\pm|\xi|)|\xi|^{-2\nu+1+j}d\xi \Big|\nonumber\\
 &\ls R^{n-2\nu}(1+Rd_g(x,y))^{-N}+\sum_{j=0}^{\nu-1}\Big|\int e^{iRd_g(x,y)\xi_1}m^{(j+1)}(|\xi|)|\xi|^{-2\nu+1+j}\varphi(|\xi|)d\xi \Big|\label{split}\\
 &\ls \begin{cases}R^{n-2\nu}\big(R d_g(x,y)\big)^{-n-\mu}\big(1+R d_g(x,y)\big )^{-N},\,\quad\quad\,\,\,\, \ n+\mu>0\\
 	R^{n-2\nu}\log(2+(Rd_g(x,y))^{-1})\big(1+R d_g(x,y)\big)^{-N},\,\, \ n+\mu=0\\
 	R^{n-2\nu}(1+Rd_g(x,y))^{-N},\ \ \ \ \ \quad\quad\quad\quad\quad\quad\quad\quad\ \  \ n+\mu<0
 \end{cases}\label{I3}
\end{align}
where $\varphi\in C^\infty$ vanishes near the origin but equals one near infinity. The first term in \eqref{split} follows from  the smoothness of  $a_\nu$ in \eqref{Ev} near $\xi=0$  and integration by parts. Moreover,
 \begin{align}\label{I4}|I_4|&\ls \sum_{\pm}\sum_{j=0}^{\nu-1}\Big| R\iiint (1-\rho(t))(tR)^{-N} m^{(N+j+1)}(s)e^{-itRs}e^{id_g(x,y)\xi_1\pm it|\xi|}\phi_{j\nu}(|\xi|)d\xi dsdt\Big|\nonumber\\
 	&\ls\ R^{-N+1}\iint (1+||\xi|-R|s||)^{-N_1}(1+|s|)^{-N+\mu}dsd\xi\nonumber\\
 &\ls R^{-N}\int (1+|\xi|/R)^{-N+\mu}d\xi\nonumber\\
 	&\ls R^{-N+n}.
\end{align}
 The  remainder term $R_{N_0}$ in \eqref{costp} is easy to handle. Indeed, for $n+\mu<N\le N_0-n-2$, using \eqref{err} we integrate by parts to obtain
  \begin{align}\label{R}\Big|\frac R{2\pi}\int \rho(t)\hat m(tR) R_{N_0}(t,x,y)dt\Big|&\ls R^{-N+1}\Big|\iint \rho(t)t^{-N}R_{N_0}(t,x,y)m^{(N)}(s)e^{-itRs}dsdt\Big|\nonumber\\
  	&\ls R^{-N+1}\int (1+R|s|)^{-N}(1+|s|)^{\mu-N} ds\nonumber\\
  	&\ls R^{-N+1}.
\end{align}
  To handle the second term in \eqref{intsum}, we notice that for $\la\ge0$
  \begin{align*}
  	\Big|\frac R{2\pi}\int (1-\rho(t))\hat m(tR)\cos (t\la)dt\Big|&\ls \Big|R\iint (1-\rho(t))(tR)^{-N} m^{(N)}(s)e^{-itRs}\cos (t\la)dtds\Big|\\
  	&\ls R^{-N+1}\int (1+|\la-R|s||)^{-N_1}(1+|s|)^{-N+\mu}ds\\
  	&\ls R^{-N}(1+\la /R)^{-N+\mu}.
  \end{align*}
  Thus, we obtain
   \begin{align}\label{2nd}
  	\Big|\frac R{2\pi}\int (1-\rho(t))\hat m(tR)\cos (tP)(x,y)dt\Big|
  	&\ls R^{-N}\sum_j (1+\la_j/R)^{-N+\mu}|e_j(x)e_j(y)|\nonumber\\
  	&\ls R^{-N}\sum_k (1+k/R)^{-N+\mu}\sum_{\la_j\in[k,k+1)}|e_j(x)e_j(y)|\nonumber\\
  	&\ls R^{-N}\sum_k (1+k/R)^{-N+\mu}(1+k)^{n-1}\nonumber\\
  	&\ls R^{-N+n}.
  \end{align}
 Here we used the $L^\infty$ bound of Laplace eigenfunctions (see e.g. \cite[Lemma 4.2.4]{fio})
 \[\sum_{\la_j\in[k,k+1)}|e_j(x)e_j(y)|\ls \sup_{x\in M}\sum_{\la_j\in[k,k+1)}|e_j(x)|^2\ls (1+k)^{n-1}.\]
 Combining the bounds \eqref{I1}, \eqref{I2}, \eqref{I3}, \eqref{I4}, \eqref{R}, \eqref{2nd}, we complete the proof.

	\section{Interior Eigenfunction estimates}

In this section, we prove the eigenfunction estimates in Theorem \ref{thm2}. We just need to prove Lemma \ref{harmonic}, and then Theorem \ref{thm2} follows from the $L^p$ bounds in Lemma \ref{lpnorm}. To proceed,  we shall use the following lemma.
			\begin{lemma}\label{diri}
		For any $f\in H^{1/2}(\partial\Omega)$, let $u\in H^1(\Omega)$ be the weak solution to the Dirichlet boundary
		value problem \eqref{harm}. Then there exists a constant $C>0$ such that 
		\begin{equation}
			\|u\|_{L^2(\Omega)}\le C\|f\|_{H^{-1/2}(\partial\Omega)}.\end{equation}
	\end{lemma}
This lemma was proved in \cite[Proposition 2.17]{CP19}. It follows from the trace theorem and standard regularity estimates (see e.g. \cite[Theorem 1.5.1.2, Theorem 1.5.1.3, Corollary 2.2.2.4, Corollary 2.2.2.6]{gris}).
\begin{lemma}\label{pdo0}
	Let $Q\in OPS^0$. Then $Q$ is bounded on $L^p$ for $1<p<\infty$, i.e.
	\[\|Qf\|_{L^p}\le C\|f\|_{L^p}.\]
\end{lemma}
Here the $L^p$ norm can be taken on $\mathbb{R}^n$ and compact manifolds. See e.g. \cite[Theorem 3.1.6, Theorem 4.3.1]{fio} for the proofs. 
\[\]

\noindent \textbf{Proof of Theorem \ref{thm2}.} It suffices to consider two cases, $p=\infty$ and $p<\infty$. 

\noindent \textbf{Case 1}: $p=\infty$. In this case, from the maximal principle (see e.g. \cite[Theorem 8.1]{GTbook}), since $e_\la$ is harmonic in $\Omega$. We get
\begin{equation}\label{linftybound}
	\|e_\la\|_{L^\infty(\Omega)} \ls \|e_\la\|_{L^\infty(\partial\Omega)}
\end{equation}
And since $V\in L^\infty(M)$, by Lemma \ref{lpnorm}, we have
	\[\|e_\la\|_{L^\infty(\partial\Omega))}\ls \la^{\frac{n-1}{2}}\|e_\la\|_{L^2(\partial\Omega))},\]
	which yields \eqref{intlp} for the case $p=\infty$.
	
\noindent \textbf{Case 2}: $p<\infty$.
In this case, let us fix a Littlewood-Paley bump function
	$\beta\in C^\infty_0((1/2,2))$ satisfying
	$$
	\sum_{\ell=-\infty}^\infty \beta(2^{-\ell} s)=1, \quad s>0.$$
	And define 
	$$\beta_0(s)=1-\sum_{\ell> 0} \beta(2^{-\ell} |s|),\,\,\,\beta_\ell(s)=\beta(2^{-\ell} |s|),\,\,\,\text{for}\,\,\,\ell>0.
	$$
Let $P=\sqrt{-\Delta_g}$. Then we have for $\ell\ge0$,
\begin{equation}\label{beta}
   \| \beta_\ell(P)f\|_{L^p(\partial\Omega))}\ls \|f\|_{L^p(\partial\Omega))}  ,\,\,\,1\le p\le \infty.
\end{equation} The implicit constant is indepdendent of $\ell$. Indeed, by Lemma \ref{pdo} we have the kernel estimates
\[|\beta_\ell(P)(x,y)|\ls 2^{n\ell}(1+2^\ell d_g(x,y))^{-N}.\] Then \eqref{beta} follows from Young's inequality.

Let $T_H$ be the harmonic extension operator from $\partial\Omega$ to $\Omega$. Then by Lemma \ref{diri}, we have
	\begin{equation}\label{sobolevbound}
		\|T_H(\beta_\ell(P)f)\|_{L^2(\Omega)} \ls \|\beta_\ell(P)f\|_{H^{-1/2}(\partial\Omega)} \ls 2^{-\ell/2} \|f\|_{L^2(\partial\Omega)}.
	\end{equation}
And from the maximal principle and \eqref{beta}, we have 
\begin{equation}\label{linftybound1}
	\|T_H(\beta_\ell(P)f)\|_{L^\infty(\Omega)} \ls \|\beta_\ell(P)f\|_{L^\infty(\partial\Omega)}\ls \|f\|_{L^\infty(\partial\Omega)}.
\end{equation}
By \eqref{sobolevbound}, \eqref{linftybound1} and interpolation, we have the following $L^p$ estimate of the frequency-localized harmonic extension operator
\begin{equation}\label{lpbound1}
	\|T_H(\beta_\ell(P)f)\|_{L^p(\Omega)} \ls 2^{-\frac{\ell}{p}}\|f\|_{L^p(\partial\Omega)}, \ 2 \leq p \leq \infty.
\end{equation}
Thus, if $2^\ell \gs \la$, we have 
\begin{equation}\label{lpbound2}
	\|T_H(\sum_{2^\ell\gs\la}\beta_\ell(P)e_\la)\|_{L^{p}(\Omega)} \ls \sum_{2^\ell\gs\la}2^{-\frac{\ell}{p}}\|e_\la\|_{L^{p}(\partial\Omega)} \ls \la^{-1/p}\|e_\la\|_{L^{p}(\partial\Omega)}.
\end{equation}
So it remains to consider $2^\ell\ls \la$. Let $\tilde \beta \in C_0^\infty$ with $\tilde\beta\equiv 1$ in a neighborhood of $(1/2,2)$ and define $\tilde\beta_\ell(s)=\tilde\beta(2^{-\ell}|s|)$. Then by \eqref{lpbound1}
\begin{equation}\label{lpbound3}
	\|T_H(\beta_\ell(P)e_\la)\|_{L^{p}(\Omega)}=\|T_H(\beta_\ell(P)\tilde \beta_\ell(P)e_\la)\|_{L^{p}(\Omega)} \ls 2^{-\frac{\ell}{p}}\|\tilde \beta_\ell(P)e_\la\|_{L^{p}(\partial\Omega)}.
\end{equation}
Moreover, for $2\le p<\infty$
\begin{equation}\label{lpbound4}
	\begin{aligned}
	\|&\tilde \beta_\ell(P)e_\la\|_{L^{p}(\partial\Omega)}\\
		&=(1+\la)^{-1} \|\tilde \beta_\ell(P)(1+\sqrt{-\Delta_g}+P_0+V)e_\la\|_{L^{p}(\partial\Omega)}\\
		&\ls (1+\la)^{-1} \|\tilde \beta_\ell(P)(1+\sqrt{-\Delta_g})e_\la\|_{L^{p}(\partial\Omega)}+ (1+\la)^{-1}\|\tilde \beta_\ell(P)(P_0+V)e_\la\|_{L^{p}(\partial\Omega)} \\
		&\ls (1+\la)^{-1} 2^{\ell}\|e_\la\|_{L^{p}(\partial\Omega)}+ (1+\la)^{-1}\|e_\la\|_{L^{p}(\partial\Omega)}
	\end{aligned}
\end{equation}
where we used \eqref{beta}, Lemma \ref{pdo0},  and the fact that $V\in L^\infty$. Using \eqref{lpbound3} and \eqref{lpbound4}, we have 
\begin{equation}\label{lpbound5}
	\|T_H(\sum_{2^\ell\ls\la}\beta_\ell(P)e_\la)\|_{L^{p}(\Omega)} \ls \sum_{2^\ell\ls\la}(1+\la)^{-1} 2^{\ell}2^{-\frac{\ell}{p}}\|e_\la\|_{L^{p}(\partial\Omega)} \ls \la^{-1/p}\|e_\la\|_{L^{p}(\partial\Omega)}.
\end{equation}
So we obtain \eqref{lplp} in Lemma \ref{harmonic}. Using the $L^p$ bounds in Lemma \ref{lpnorm}, we complete the proof of Theorem \ref{thm2}.

\section{Measure of nodal set }
In this section, we prove the nodal set estimates in Theorem \ref{thm1}.

First, we establish some general results for Sobolev spaces on compact manifolds. These results will be used to prove the regularity of eigenfunctions.  They are likely to be useful for future research, so we give detailed proofs for them.

Let $s>0$ and $1<p< \infty$. We can define the Sobolev norm on $M$ by local coordinates
\begin{equation}\label{norm1}
	\|f\|_{W^{s,p}(M)}=\sum_\nu\|(I-\Delta)^{s/2}f_\nu\|_{L^p(\mathbb{R}^n)}.
\end{equation}
where $f_\nu=(\phi_\nu f)\circ\kappa_\nu ^{-1}$, and $\{\phi_\nu\}$ is a partition of unity subordinate to a finite covering $M=\cup \Omega_\nu$, and $\kappa_\nu:\Omega_\nu\to \tilde \Omega_\nu\subset \mathbb{R}^n$ is the coordinate map. For simplicity, we sometimes do not distinguish between  $\Omega_{\nu}$ and $\tilde \Omega_\nu$, $f_\nu$ and $\phi_\nu f$, since they are identical up to the coordinate map.

Moreover, we can also define another Sobolev norm by pseudo-differential operators
\begin{equation}\label{norm2}
	\|f\|_{H^{s,p}(M)}=\|(I-\Delta_g)^{s/2}f\|_{L^p(M)}.
\end{equation}
By \cite[Theorem 4.3.1]{fio}, we see that $(I-\Delta_g)^{s/2}$ is an invertible pseudo-differential operator of order $s$ with elliptic principal symbol $(\sum g^{jk}(x)\xi_j\xi_k)^{s/2}$. Moreover, if we replace  $(I-\Delta_g)^{s/2}$ in \eqref{norm2} by any invertible pseudo-differential operator of order $s$, then it still gives a comparable norm, by Lemma \ref{pdo0}.

We  prove that these two Sobolev norms are equivalent.
\begin{proposition}\label{sobolev}
	For $s>0$ and $1<p<\infty$, we have
	\[\|f\|_{W^{s,p}(M)}\approx \|f\|_{H^{s,p}(M)}.\]
	The implicit constants are independent of $f$.
\end{proposition}
As a corollary, different partitions of unity and such coordinate atlases in the definition \eqref{norm1} give comparable norms. When $p=2$, Proposition \ref{sobolev} follows from Plancherel theorem and the $L^2$-boundedness of zero order pseudo-differential operators, see e.g. \cite[section 4.2]{hangzhou}. The case $p\ne2$ is more complicated, and it is very difficult to find good references. To prove this on our own, we start with the following key lemma. Roughly speaking, this lemma establishes a ``linear relation'' between any two pseudo-differential operators of the same order.

\begin{lemma}\label{localinverse}
	Let $s>0$. Let $V_1,\  V,\ \Omega$ be open sets such that $\bar{V}_1 \subset V \subset \Omega$. Let $P_1,\  P\in OPS^s$ with symbols supported in $V_1,\ V$ respectively.  If the principal symbol $\bar p(x,\xi)$ of $P$ is elliptic on $\bar V_1$, i.e., for any $ x\in \bar V_1,$
	\[\bar p(x,\xi)\ne0,\ \forall \xi\ne0,\]
	then there is a $Q\in OPS^0$  with symbol supported in $V_1$ such that
	\begin{equation}\label{Qlocal}
		P_1 - QP \in OPS^0.
	\end{equation}
\end{lemma}
\begin{proof}
	Let $p_1(x,\xi)$ be the symbols of $P_1$ on $\Omega$. Since $\bar p(x,\xi)$ is elliptic on the support of $p_1(x,\xi)$, we have 
	\begin{equation}
		\frac{\varphi(\xi) p_1(x,\xi)}{\bar{p}(x,\xi)}\in S^0
	\end{equation}
where $\varphi\in C^\infty$ vanishes near the origin but equals one near infinity.	Denote the associated zero order pseudo-differential operator by  $Q_0$. Let $R_{-1} = P_1 - Q_0P$. Then by the Kohn-Nirenberg theorem (see e.g. \cite[Theorem 3.1.1]{fio}), we have $R_{-1}\in OPS^{s-1}$. The symbol of $R_{-1}$ is supported in $V_1$. 
If $s\le1$, then we are done by setting $Q=Q_0$, since $P_1-Q_0P\in OPS^{s-1}\subset OPS^0$.

Next, it remains to consider $s>1$. Let $k = \lceil s\rceil\ge2$. We need to construct $Q_{-i} \in OPS^{-i},\ R_{-i-1} \in OPS^{s-i-1}$ recursively for $1\le i\le k-1$. If $r_i(x,\xi)$ is the symbol of $R_{-i}$, and  $Q_{-i}$ has the symbol
	\begin{equation}
		\frac{\varphi(\xi) r_i(x,\xi)}{\bar{p}(x,\xi)}\in S^{-i},
	\end{equation}
then using the Kohn-Nirenberg theorem we have $R_{-i-1} = R_{-i} - Q_{-i}P\in OPS^{s-i-1}$. The symbol of $R_{-i-1}$ is  supported in $V_1$.
Let $$Q = \sum_{i=1}^{k-1} Q_{-i}.$$  The symbol of $Q$ is supported in $V_1$. Then $ P_1 - QP=R_{-k} \in OPS^{s-k} \subset OPS^0$.
\end{proof}

\noindent \textbf{Proof of Proposition \ref{sobolev}.}
	The basic idea is to verify these two equivalences
	\begin{equation}\label{idea}\|(I-\Delta_g)^{s/2}f\|_{L^p(M)} \approx \sum_\nu \|(I-\Delta_g)^{s/2}f_\nu\|_{L^p(M)}\approx \sum_\nu \|(I-\Delta)^{s/2}f_\nu\|_{L^p(\mathbb{R}^n)}.\end{equation}
	The first equivalence is straightforward. Indeed,
	The relation $\ls$ follows from Minkowski inequality. And for the other direction, we use Lemma \ref{pdo0} to see that
	\begin{equation} \label{splitls}
		\begin{aligned}
			\|(I-\Delta_g)^{s/2}f_\nu\|_{L^p(M)} &= \|(I-\Delta_g)^{s/2}M_{\phi_\nu}(I-\Delta_g)^{-s/2}((I-\Delta_g)^{s/2}f)\|_{L^p(M)}\\
			&\ls \|(I-\Delta_g)^{s/2}f\|_{L^p(M)},
		\end{aligned}
	\end{equation}
where $M_{\phi_\nu}$ stands for the operator of multiplying by $\phi_\nu(x)$. 
	Summing up of \eqref{splitls}  over $\nu$ we obtain the first equivalence in \eqref{idea}.
	
	To prove the second equivalence in \eqref{idea}, it suffices to show that for each $\nu$
	\begin{equation}\label{equi2}
		\|(I-\Delta_g)^{s/2}f_\nu\|_{L^p(M)}\approx \|(I-\Delta)^{s/2}f_\nu\|_{L^p(\mathbb{R}^n)}.
	\end{equation}

		For each $\Omega_\nu,\ \phi_\nu\in C_0^\infty(\Omega_\nu)$ in \eqref{norm1}, we can find open subsets $V_\nu,\ U_\nu,\ W_\nu$ of $\Omega_\nu$, and cutoff functions $\psi_\nu \in C_0^\infty(V_\nu),\ \psi_{\nu1} \in C_0^\infty(U_\nu),\  \psi_{\nu2} \in C_0^\infty(W_\nu),\  \eta_\nu \in C_0^\infty(\Omega_{\nu})$ such that
$$
		\supp\phi_\nu \subset\subset U_{\nu} \subset V_\nu \subset\subset W_\nu
$$
	and   $\psi_{\nu} \equiv 1$ on $\bar{U}_\nu$, $\psi_{\nu2} \equiv  1$ on $\bar{V}_{\nu}$,  $\eta_\nu \equiv  1$ on $\bar{W}_\nu$.
	
	Let $P_\nu=\psi_{\nu}(I-\Delta)^{s/2},\   P_{\nu1}=\psi_{\nu1}(I-\Delta_g)^{s/2}M_{\eta_\nu}$. We see that $M_{\eta_\nu}\in OPS^0$, and $ P_\nu, P_{\nu1}\in OPS^s$. Note that the principal symbol of $P_\nu$ is $\psi_\nu(x)|\xi|^s$, which is elliptic on $\bar{U}_\nu$.  By Lemma \ref{localinverse}, we can find $Q_{\nu1} \in OPS^0$ supported in $U_\nu$ such that
$$
		P_{\nu1} - Q_{\nu1}P_{\nu} \in OPS^0.
$$
	Then by Lemma \ref{pdo0} we obtain the local estimate
	\begin{equation}\label{small1}
		\|P_{\nu1}(f_\nu)\|_{L^p(\Omega_\nu)} = \|( P_{\nu1} - Q_{\nu1}P_{\nu})(f_\nu) + Q_{\nu1}P_\nu(f_\nu)\| _{L^p(\Omega_\nu)}
		\ls \|f_\nu\|_{L^p(\Omega_\nu)} + \|P_\nu f_\nu\|_{L^p(\Omega_\nu)}.
	\end{equation}

	Moreover, if $ P_{\nu2}=\psi_{\nu2}(I-\Delta_g)^{s/2}M_{\eta_\nu}$, then $P_{\nu2}$ has the principal symbol $\psi_{\nu2}(x)(\sum g^{jk}(x)\xi_j\xi_k)^{s/2}$, which is elliptic on $\bar{V}_\nu$.
	 Similarly, by applying Lemma \ref{localinverse} to $P_\nu$ and $P_{\nu2}$, we obtain the local estimate
	\begin{equation}\label{large1}
		\|P_\nu(f_\nu)\|_{L^p(\Omega_\nu)} \ls \|f_\nu\|_{L^p(\Omega_\nu)} + \| P_{\nu2} f_\nu\|_{L^p(\Omega_\nu)}.
	\end{equation}
	
	Next, we handle the nonlocal part.
	We write
	\begin{equation}
		(1-\psi_\nu)(I-\Delta)^{s/2}f_\nu = (1-\psi_\nu)(I-\Delta)^{s/2}(\phi_\nu\eta_\nu f)
		= (1-\psi_\nu)(I-\Delta)^{s/2}M_{\phi_\nu}(\eta_\nu f).
	\end{equation}
	Since dist($\supp(1-\psi_\nu),\ \supp \phi_\nu ) = \delta_\nu > 0$,	using integration by parts, we see that the kernel of $(1-\psi_\nu)(I-\Delta)^{s/2}M_{\phi_\nu}$ satisfies
\[
		\Big|\int_{\mathbb{R}^n} (1-\psi_\nu(x))e^{i(x-y)\cdot \xi}\phi_\nu(y)(1+|\xi|^2)^{s/2}d\xi\Big|
		\ls (1+|x-y|)^{-N},\ \forall N.
\]
By Young's inequality, we get
	\begin{equation}\label{localeuc}
		\|(1-\psi_\nu)(I-\Delta)^{s/2}(f_\nu))\|_{L^p(\mathbb{R}^n)} \ls \|\eta_\nu f\|_{L^p(\mathbb{R}^n)} = \| f_\nu\|_{L^p(\Omega_\nu)}.
	\end{equation}
	
	Similarly, using the fact that the kernel of pseudo-differential operators on compact manifolds is smooth away from diagonal, we have
	\begin{equation}\label{small2}
		\|(1-\psi_{\nu1})(I-\Delta_g)^{s/2}(f_\nu)\|_{L^p(M)} = \|(1-\psi_{\nu1})(I-\Delta_g)^{s/2}(\phi_\nu\eta_\nu f)\|_{L^p(M)}
		\ls \|\eta_\nu f\|_{L^p(M)} = \| f_\nu\|_{L^p(\Omega_\nu)}
	\end{equation}
	and
	\begin{equation}\label{large2}
		\|(1-\psi_{\nu2})(I-\Delta_g)^{s/2}(f_\nu)\|_{L^p(M)} \ls \| f_\nu\|_{L^p(\Omega_\nu)}.
	\end{equation}
	
	Combining \eqref{small1} with the nonlocal estimates \eqref{localeuc} and \eqref{small2}, we obtain
	\begin{equation}\label{small}
		\begin{aligned}
			\|(I-\Delta_g)^{s/2}f_\nu\|_{L^p(M)} &\ls \|(1-\psi_{\nu1})(I-\Delta_g)^{s/2}f_\nu\|_{L^p(M)} + \|\psi_{\nu1}(I-\Delta_g)^{s/2}f_\nu\|_{L^p(M)} \\
			&\ls \|f_\nu\|_{L^p(\Omega_\nu)} + \|\psi_{\nu1}(I-\Delta_g)^{s/2}f_\nu\|_{L^p(\Omega_\nu)}\\
			&= \|f_\nu\|_{L^p(\Omega_\nu)} + \| P_{\nu1}f_\nu\|_{L^p(\Omega_\nu)}\\
			&\ls \|f_\nu\|_{L^p(\Omega_\nu)} + \|P_{\nu}f_\nu\|_{L^p(\Omega_\nu)}\\
			&=  \|f_\nu\|_{L^p(\Omega_\nu)} + \|(I-\Delta)^{s/2}f_\nu - (1-\psi_{\nu})(I-\Delta)^{s/2}f_\nu\|_{L^p(\Omega_\nu)}\\
			&\ls  \|f_\nu\|_{L^p(\Omega_\nu)} + \|(I-\Delta)^{s/2}f_\nu\|_{L^p(\mathbb{R}^n)}\\
			&\ls \|(I-\Delta)^{s/2}f_\nu\|_{L^p(\mathbb{R}^n)}.
		\end{aligned}
	\end{equation}
Here in the last step we apply Lemma \ref{pdo0} to $(I-\Delta)^{-s/2}\in OPS^0$.

Similarly,	combining \eqref{large1} with the nonlocal estimates \eqref{localeuc} and \eqref{large2}, we have
\begin{equation}\label{large}
	\begin{aligned}
		\|(I-\Delta)^{s/2}f_\nu\|_{L^p(\mathbb{R}^n)} &\ls \|(1-\psi_{\nu})(I-\Delta)^{s/2}f_\nu\|_{L^p(\mathbb{R}^n)} + \|\psi_{\nu}(I-\Delta)^{s/2}f_\nu\|_{L^p(\mathbb{R}^n)} \\
		&\ls \|f_\nu\|_{L^p(\Omega_\nu)} + \|\psi_{\nu}(I-\Delta)^{s/2}f_\nu\|_{L^p(\Omega_\nu)}\\
		&= \|f_\nu\|_{L^p(\Omega_\nu)} + \| P_{\nu}f_\nu\|_{L^p(\Omega_\nu)}\\
		&\ls \|f_\nu\|_{L^p(\Omega_\nu)} + \|P_{\nu2}f_\nu\|_{L^p(\Omega_\nu)}\\
		&=  \|f_\nu\|_{L^p(\Omega_\nu)} + \|(I-\Delta_g)^{s/2}f_\nu - (1-\psi_{\nu2})(I-\Delta_g)^{s/2}f_\nu\|_{L^p(\Omega_\nu)}\\
		&\ls  \|f_\nu\|_{L^p(\Omega_\nu)} + \|(I-\Delta_g)^{s/2}f_\nu\|_{L^p(M)}\\
		&\ls\|(I-\Delta_g)^{s/2}f_\nu\|_{L^p(M)}.
	\end{aligned}
\end{equation}
In the last step we used Lemma \ref{pdo0} for $(I-\Delta_g)^{-s/2}\in OPS^0$. So we finish the proof of \eqref{equi2}. Thus, the proof of Proposition \ref{sobolev} is complete.

Let $[\D,V]=\D V-V\D$. We need to following commutator estimate.
\begin{lemma}\label{taylor}
	Let $1<p<\infty$. Given $P\in OPS^1$,
	\[\|[P, f] u\|_{L^p} \le C\|f\|_{{\rm Lip}^1}\|u\|_{L^p}.\]
	Here $\|f\|_{{\rm Lip}^1}$ is the Lipschitz norm of $f$.
\end{lemma}
Here the $L^p$ norm can be taken on $\mathbb{R}^n$ and compact manifolds. See Proposition 1.3 in Taylor \cite{taylor}. The result was proven in Calder\'on \cite{cald65} for classical first-order pseudodifferential	operators and by  Coifman-Meyer \cite{cm} for $OPS^1$.

 \begin{lemma}\label{C1} If $V\in Lip^1(\M)$, then $e_\la \in C^{1,\alpha}(M)$, for any $0< \alpha < 1$. 
\end{lemma}
\begin{proof}
	By Sobolev imbedding (see e.g. \cite{evans}), we only need to show $\|e_\la\|_{W^{2,p}(M)} < \infty$ for any $1<p<\infty$. Indeed, using the commutator estimate in Lemma \ref{taylor} and the equation $(\D+V)e_\la=\la e_\la$, we have
\begin{align*}
	\|\D (Ve_\la)\|_{L^p(M)}&\le \|V(\D+V) e_\la\|_{L^p(M)}+\|V^2e_\la\|_{L^p(M)}+\|[\D,V]e_\la\|_{L^p(M)}\\
	&\ls \la\|V\|_{L^\infty}\|e_\la\|_{L^p(M)}+\|V\|_{L^\infty}^2\|e_\la\|_{L^p(M)}+\|V\|_{{\rm Lip}^1}\|e_\la\|_{L^p(M)}\\
	&\ls (1+\la) \|e_\la\|_{L^p(M)}.
\end{align*}
So by Proposition \ref{sobolev}, we obtain
		\begin{align*}
			\|e_\la\|_{W^{2,p}(M)} &\approx \|(1+\D)^2 e_\la\|_{L^p(M)}\\
			&\ls  \|(1+\D)(1+\D+V) e_\la\|_{L^p(M)} +  \|\D (Ve_\la)\|_{L^p(M)} + \|V\|_{L^\infty}\|e_\la\|_{L^p(M)}\\
			&\ls (1+\la)(\|(1+\D)e_\la\|_{L^p(M)}+\|e_\la\|_{L^p(M)})\\
			&\le (1+\la)(\|(1+\D+V)e_\la\|_{L^p(M)}+\|V\|_{L^\infty}\|e_\la\|_{L^p(M)}+\|e_\la\|_{L^p(M)})\\
			&\ls (1+\la)^2\|e_\la\|_{L^p(M)}.
		\end{align*}
\end{proof}
Next, we prove the nodal set estimates. Let 
\[N_\la=\{x\in \M: e_\la(x)=0\},\]
\[D_+=\{x\in \M: e_\la(x)>0\},\]
\[D_-=\{x\in \M: e_\la(x)<0\}.\]
We have $\partial D_\pm=N_\la$. We first express the manifold $M$ as a (essentially) disjoint union 
\[M=\bigcup_{j\ge1}D_{j,+}\cup \bigcup_{j\ge1}D_{j,-}\cup N_\la \]where $D_{j,+}$ and $D_{j,-}$ are are the positive and negative nodal domains of $e_\la$, i.e, the connected
components of the sets $D_+$ and $D_-$. For simplicity, we assume that there are only two nodal domains $D_+$ and $D_-$. Since $\nabla e_\la$ is continuous by Lemma \ref{C1} and we are assuming that zero is a regular value of $e_\la$,  we can apply Gauss-Green theorem on each nodal domain $D_{\pm}$ with boundary $\partial D_{\pm}$.  We have
\[\int_{D_+}div (f\nabla e_\la)dV_g=\int_{N_\la}\langle f\nabla e_\la,\nu_-\rangle dS=-\int_{N_\la}f|\nabla e_\la|dS\]
\[\int_{D_-}div (f\nabla e_\la)dV_g=\int_{N_\la}\langle f\nabla e_\la,\nu_+\rangle dS=\int_{N_\la}f|\nabla e_\la|dS.\]

\begin{equation}\label{green}
	2\int_{N_\la}f|\nabla e_\la|=\int_{D_-}div (f\nabla e_\la)-\int_{D_+}div (f\nabla e_\la).
\end{equation}
Note that by Cauchy-Schwarz\[\int_{N_\la}|\nabla e_\la|\ls (\int_{N_\la}|\nabla e_\la|^2)^\frac12|N_\la|^\frac12.\]
So to estimate the lower bound of $|N_\la|$, it suffices to estimate $\int_{N_\la}|\nabla e_\la|$ and $\int_{N_\la}|\nabla e_\la|^2$. 
\begin{lemma}\label{1norm} If $V\in Lip^1(\M)$, then
	\[\int_{N_\la}|\nabla e_\la|\ge \frac{\la^2}4\|e_\la\|_{L^1(\M)}.\]
\end{lemma}
\begin{lemma}\label{2norm}If $V\in Lip^1(\M)$, then
	\[\int_{N_\la}|\nabla e_\la|^2\ls \la^3\|e_\la\|_{L^2(\M)}.\]
\end{lemma}

Using the these two lemmas and the eigenfunction estimate \eqref{low},
we get the lower bound of the nodal set in Theorem \ref{thm1}
\[|N_\la|\gs \la^{\frac{3-n}2}.\]

\subsection{Proof of Lemma \ref{1norm}} We set $f=1$ in \eqref{green}. So
\[2\int_{N_\la}|\nabla e_\la|=\int_{D_-}\Delta_g e_\la-\int_{D_+}\Delta_g e_\la.\]
Since $\sqrt{-\Delta_g}=\D-P_0$, we have
\[-\Delta_g=(\D+V)^2-(\D V-V\D)-2V(\D +V)+V^2-2P_0(\D+V)+2P_0V+Q_0,\]
where $Q_0=P_0\D-\D P_0+P_0^2\in OPS^0$. Thus,
\begin{align*}
	2\int_{N_\la}|\nabla e_\la|&=\int_{D_+}-\int_{D_-} (\la^2 e_\la-[\D,V]e_\la-2\la Ve_\la+V^2e_\la-2\la P_0 e_\la+2P_0 Ve_\la+Q_0e_\la)\\
	&\ge \la^2\|e_\la\|_{L^1(\M)}-\|[\D, V]e_\la\|_{L^1(\M)}-2\la\|Ve_\la\|_{L^1(\M)}-\|V^2e_\la\|_{L^1(\M)}\\
	&\quad\quad\quad\quad\quad\quad-2\la\|P_0 e_\la\|_{L^1(\M)}-2\|P_0Ve_\la\|_{L^1(\M)}-\|Q_0 e_\la\|_{L^1(\M)}.
\end{align*}
By H\"older's inequality and \eqref{low}, we have
\[\|e_\la\|_{L^{1+\eps}(\M)}\ls \la^{\frac{(n-1)\eps}{2(1+\eps)}}\|e_\la\|_{L^1(\M)},\ 0<\eps<1.\]
Combining this estimate with Lemma \ref{taylor}, we have
\[\|[\D, V]e_\la\|_{L^1(\M)}\ls \|[\D, V]e_\la\|_{L^{1+\eps}(\M)}\ls \|V\|_{{\rm Lip}^1}\|e_\la\|_{L^{1+\eps}(\M)}\ls \la\|V\|_{{\rm Lip}^1}\|e_\la\|_{L^{1}(\M)},\]
if $\eps>0$ is small enough.
Moreover, if $\eps>0$ is small enough, then by Lemma \ref{pdo0} we have
\[\la\|Ve_\la\|_{L^1(\M)}\ls \la\|V\|_{L^\infty}\|e_\la\|_{L^1(\M)}\]
\[\|V^2e_\la\|_{L^1(\M)}\ls \|V\|_{L^\infty}^2\|e_\la\|_{L^1(\M)}\]
\[\la\|P_0 e_\la\|_{L^1(\M)}\ls \la\|P_0e_\la\|_{L^{1+\eps}(\M)}\ls \la\|e_\la\|_{L^{1+\eps}(\M)}\ls \la^{\frac32}\|e_\la\|_{L^1(\M)}\]
\[\|P_0Ve_\la\|_{L^1(\M)}\ls \|P_0Ve_\la\|_{L^{1+\eps}(\M)}\ls \|Ve_\la\|_{L^{1+\eps}(\M)}\ls \la\|V\|_{L^\infty}\|e_\la\|_{L^1(\M)}\]
\[\|Q_0 e_\la\|_{L^1(\M)}\ls \|Q_0 e_\la\|_{L^{1+\eps}(\M)}\ls \| e_\la\|_{L^{1+\eps}(\M)}\ls \la\| e_\la\|_{L^1(\M)}.\]
So we finish the proof Lemma \ref{1norm}.

\subsection{Proof of Lemma \ref{2norm}}		
We set $f=\sqrt{1+|\nabla e_\la|^2}$ in \eqref{green}. And then
\begin{align*}
	2\int_{N_\la}\sqrt{1+|\nabla e_\la|^2}\ |\nabla e_\la|&=\int_{D_-}div(\sqrt{1+|\nabla e_\la|^2}\nabla e_\la)-\int_{D_+}div(\sqrt{1+|\nabla e_\la|^2}\nabla e_\la)\\
	&\ls\int_{\M}|div(\sqrt{1+|\nabla e_\la|^2}\nabla e_\la)|\\
	&\ls \int_{\M}\sqrt{1+|\nabla e_\la|^2}\ |\nabla^2 e_\la|\\
	&\ls (\|e_\la\|_{L^2(\M)}+\|\nabla e_\la\|_{L^2(M)})\|\nabla^2e_\la\|_{L^2(\M)}\\
	&\ls \la^3\|e_\la\|_{L^2(\M)}.
\end{align*}
Here we use the Sobolev estimates of eigenfunctions in the last step. Indeed, we have the following Sobolev estimates
\begin{align*}
	\|\nabla e_\la\|_{L^2(\M)}&\ls \|\D e_\la\|_{L^2(\M)}+\|e_\la\|_{L^2(\M)}\\
	&\le \|(\D+V)e_\la\|_{L^2(\M)}+\|Ve_\la\|_{L^2(\M)}+\|e_\la\|_{L^2(\M)}\\
	&\ls\la\|e_\la\|_{L^2(\M)}+\|V\|_{L^\infty}\|e_\la\|_{L^2(\M)}\\
	&\ls \la\|e_\la\|_{L^2(\M)},
\end{align*}
and similarly, we may exploit Lemma \ref{taylor} to obtain
\begin{align*}\|\nabla^2e_\la\|_{L^2(\M)}&\ls \|\D^2 e_\la\|_{L^2(\M)}+\|\D e_\la\|_{L^2(\M)}+\|e_\la\|_{L^2(\M)}\\
	&\ls \|(\D+V)^2e_\la\|_{L^2(\M)}+\|[\D,V]e_\la\|_{L^2(\M)}+\|V(\D+V)e_\la\|_{L^2(\M)}+\\
	&\ \ \ \ \ \|V^2e_\la\|_{L^2(\M)}+\la\|e_\la\|_{L^2(\M)}\\
	&\ls \la^2\|e_\la\|_{L^2(\M)}+\la\|V\|_{{\rm Lip}^1}\|e_\la\|_{L^2(\M)}+\|V\|_{L^\infty}^2\|e_\la\|_{L^2(\M)}\\
	&\ls \la^2\|e_\la\|_{L^2(\M)}.
\end{align*}
So Lemma \ref{2norm} is proved.

\section*{Acknowledgements}
Y.S. is partially supported by the NSF DMS Grant $2154219$. C.Z. is partially supported by a startup grant from Tsinghua University.

		\bibliographystyle{plain}

\begin{thebibliography}{0}
			
			
			\bibitem{BSS19} M. Blair, Y. Sire, C. Sogge, Quasimode, eigenfunction and spectral projection bounds for Schr\"odinger operators on manifolds with critically singular potentials,  J. Geom. Anal. 31 (2021), no. 7, 6624-6661. 
			\bibitem{blair2021uniform}
M. D. Blair, X. Huang,  Y. Sire and C. D. Sogge,
Uniform Sobolev estimates on compact manifolds involving singular potentials.
Rev. Mat. Iberoam. 38 (2022), no. 4, 1239-1286.
			\bibitem{bj07} K. Bogdan and T. Jakubowski, Estimates of heat kernel of fractional Laplacian pertubed by gradient operators, Commun. Math. Phys. 271 (2007), 179-198.
			\bibitem{BSS} K. Bogdan, A. St\'os and P. Sztonyk, Harnack inequality for stable processes on d-sets, Studia Math. 158 (2003), 163-198.
		\bibitem{bour09}J Bourgain. Geodesic restrictions and lp-estimates for eigenfunctions of riemannian surfaces. Amer.
		Math. Soc. Tranl, 226:27–25, 2009.
		\bibitem{BGT07}N. Burq, P. Gerard, and N. Tzvetkov. Restriction of the Laplace-Beltrami eigenfunctions to submanifolds.
		Duke Math. J., 138:445–486, 2007.
		\bibitem{cald65} A. Calderon, Commutators of singular integral operators, Proc. NAS, USA 53 (1965),
		1092-1099.
		\bibitem{cald} A. Calder\'on, On an inverse boundary value problem, Seminar on Numerical Analysis and its applications to Continuum Physics, Soc.Brasil.Mat.,
		Rio de Janeiro, (1980), 65–73.
			\bibitem{carmona} R. Carmona, W. Masters and  B. Simon, Relativistic Schr\"odinger operators: asymptotic behavior of the eigenfunctions. J. Funct. Anal. 91 (1990), no. 1, 117-142.
			\bibitem{cks} Z.-Q. Chen, P. Kim, and R. Song, Stability of Dirichlet heat kernel estimates for non-local operators under Feynman-Kac perturbation, Transactions of the American Mathematical Society 367.7 (2015): 5237-5270.
			\bibitem{cm}R. Coifman and Y. Meyer, Commutateurs d'integrales singulieres et operateurs multilineaires, Ann. Sci. Inst. Fourier 28 (1978), 177-202.
			\bibitem{col}B. Colbois, A. Girouard, C. Gordon, D. Sher. Some recent developments on the Steklov eigenvalue problem,  preprint arXiv:2212.12528
				\bibitem{cox}  Cox, Graham; Jakobson, Dmitry; Karpukhin, Mikhail; Sire, Yannick. Conformal invariants from nodal sets. II. Manifolds with boundary. J. Spectr. Theory 11 (2021), no. 2, 387-409.
					\bibitem{CP19}Di Cristo, Michele; Rondi, Luca Interior decay of solutions to elliptic equations with respect to frequencies at the boundary. Indiana Univ. Math. J. 70 (2021), no. 4, 1303–1334. 
			\bibitem{daubechies} I. Daubechies and E. H. Lieb, One-electron relativistic molecules with Coulomb interaction. Comm. Math. Phys. 90 (1983), no. 4, 497-510.
		\bibitem{esc1}J. Escobar. The Yamabe problem on manifolds with boundary. Jour. Diff. Geometry 35 (1992),
		21–84.
		\bibitem{esc2}J. Escobar. Conformal deformation of a Riemannian metric to a scalar flat metric with con-
		stant mean curvature on the boundary. Annals of Mathematics 136 (1992), 1–50.
		\bibitem{esc3}J. Escobar. Uniqueness and non-uniqueness of metrics with prescribed scalar and mean curva-
		ture on compact manifolds with boundary. Journal of Functional Analysis 202 (2003), 424–442.
		
		\bibitem{evans}Evans, Lawrence C. Partial differential equations. Second edition. Graduate Studies in Mathematics, 19. American Mathematical Society, Providence, RI, 2010. xxii+749
			\bibitem{fs} R.L. Frank and J. Sabin. Sharp Weyl laws with singular potentials. Preprint 	arXiv:2007.04284
			\bibitem{FLS1} R. L. Frank, E. H. Lieb and R. Seiringer, Stability of relativistic matter with magnetic fields for nuclear charges up to the critical value. Comm. Math. Phys. 275 (2007), no. 2,479-489.
			\bibitem{FLS2} R. L. Frank, E. H. Lieb and R. Seiringer,  Hardy-Lieb-Thirring inequalities for fractional Schr\"odinger operators. J. Amer. Math. Soc. 21 (2008), no. 4, 925-950.
			\bibitem{GT19}Galkowski, Jeffrey; Toth, John A. Pointwise bounds for Steklov eigenfunctions. J. Geom. Anal. 29 (2019), no. 1, 142–193.
			\bibitem{GTbook}Gilbarg, David; Trudinger, Neil S. Elliptic partial differential equations of second order. Reprint of the 1998 edition. Classics in Mathematics. Springer-Verlag, Berlin, 2001.
			\bibitem{GG} H. Gimperlein and G. Grubb, Heat kernel estimates for pseudodifferential operators, fractional
			Laplacians and Dirichlet-to-Neumann operators, J. Evol. Equ. 14 (2014), 49-83.
			\bibitem{GP17}Girouard, Alexandre; Polterovich, Iosif Spectral geometry of the Steklov problem (survey article). J. Spectr. Theory 7 (2017), no. 2, 321–359.
			\bibitem{gri03}A. Grigor’yan, Heat kernels and function theory on metric measure spaces, in: Heat kernels and
			analysis on manifolds, graphs, and metric spaces (Paris, 2002), Contemp. Math. 338, Amer. Math.
			Soc., Providence, RI, 2003, 143–172.
			\bibitem{gris}P. Grisvard, Elliptic Problems in Nonsmoooth Domains, Pitman, Boston London Melbourne,
			1985.
			\bibitem{HL01}Hislop, P. D.; Lutzer, C. V.
			Spectral asymptotics of the Dirichlet-to-Neumann map on multiply connected domains in $R^d$.
			Inverse Problems 17 (2001), no. 6, 1717–1741.
			\bibitem{jmpa}X. Huang, Y. Sire, C. Zhang, Spectral cluster estimates for Schrödinger operators of relativistic type. J. Math. Pures Appl. (9) 155 (2021), 32-61.
			\bibitem{hs}X. Huang and C. D. Sogge. Weyl formulae for Schr\" odinger operators with critically singular potentials, to appear in Comm. Partial Differential Equations
				\bibitem{hz}X. Huang and C. Zhang. Pointwise Weyl Laws for Schrodinger operators with singular potentials. arXiv:2103.05531 to appear in Adv. Math.
				\bibitem{hztorus} X. Huang and C. Zhang. Sharp Pointwise Weyl Laws for Schr\"odinger operators with singular potentials on flat tori, preprint 	arXiv:2109.13370
			\bibitem{LY} P. Li and S.T. Yau, On the parabolic kernel of the Schr\"odinger operator, Acta Math. 156 (1986), 153-201.
			\bibitem{man01}N. Mandache, Exponential instability in an inverse problem for the Schr¨odinger equation,
			Inverse Problems 17 (2001) 1435–1444.
		
			\bibitem{LiebYau1} E. H. Lieb and  H.-T. Yau,  The Chandrasekhar theory of stellar collapse as the limit of quantum mechanics. Comm. Math. Phys. 112 (1987), no. 1, 147-174.
			
			\bibitem{LiebYau2} E. H. Lieb and H.-T. Yau, The stability and instability of relativistic matter. Comm. Math. Phys. 118 (1988), no. 2, 177-213.
			\bibitem{mar}F. Marques. Conformal deformations to scalar-flat metrics with constant mean curvature on
			the boundary. Comm. in Analysis and Geometry, 15, No. 2 (2007), 381–405.
			\bibitem{PST} I. Polterovich; D.A. Sher; J. A.  Toth. Nodal length of Steklov eigenfunctions on real-analytic Riemannian surfaces. 
			J. Reine Angew. Math. 754 (2019), 17–47.
			\bibitem{sal} L. Saloff-Coste, A note on Poincar\'e, Sobolev, and Harnack inequalities, IMRN. 1992, No. 2
			\bibitem{SRV} R.L. Schilling, R. Song and Z. Vondracek, Berstein Functions, Walter de Gruyter, Berlin, 2010.
			\bibitem{ss}A. Seeger and C.D. Sogge, Bounds for eigenfunctions of differential operators, Indiana Math. J. 38(1989), 669-682.
			\bibitem{simon}	B. Simon. Schr\"odinger semigroups. Bull. Amer. Math. Soc. (N.S.), 7(3):447–526, 1982.
		 \bibitem{sogge88}C. D. Sogge. Concerning the Lp norm of spectral clusters for second-order elliptic operators
		on compact manifolds. J. Funct. Anal., 77(1):123-138, 1988.
			\bibitem{fio}C. D. Sogge. Fourier integrals in classical analysis, volume 210 of Cambridge Tracts in Mathematics.
		Cambridge University Press, Cambridge, second edition, 2017.
				\bibitem{hangzhou} C. D. Sogge. Hangzhou lectures on eigenfunctions of the Laplacian, volume 188 of Annals of Mathematics Studies. Princeton University Press, Princeton, NJ, 2014.
				\bibitem{sogge08}C. D. Sogge. Lectures on eigenfunctions of the Laplacian. Topics in mathematical analysis, 337–360,
				Ser. Anal. Appl. Comput., 3, World Sci. Publ., Hackensack, NJ, 2008.
				\bibitem{tohoku}C. D. Sogge. Kakeya-Nikodygm averages and Lp-norms of eigenfunctions. Tohoku Math. J., 63:519–538,
			2011.
			\bibitem{sogge2015}C. D. Sogge. Problems related to the concentration of eigenfunctions. Journees equations aux derivees partielles (2015), article no. 9, 11 p.
			\bibitem{swz}C. Sogge, X. Wang, and J. Zhu. Lower bounds for interior nodal sets of Steklov
			eigenfunctions. Proceedings of the American Mathematical Society, 144:4715–
			4722, 2016.		
			\bibitem{sz2011}C. D. Sogge and S. Zelditch, Lower bounds on the Hausdorff measure of nodal sets, Math. Res. Lett.
			18(2011), 25-37.
				\bibitem{song} R. Song, Feynman-Kac Semigroup with Discontinuous Additive Functionals, Journal of Theoretical Probability, Vol. 8, No. 4, 1995
					\bibitem{sturm} K.T. Sturm, Heat kernel bounds on manifolds, Math. Ann. 292(1992), 149-162.
			\bibitem{tataru}Daniel Tataru. On the regularity of boundary traces for the wave equation. Annali della Scuola Normale
			Superiore di Pisa-Classe di Scienze, 26(1):185–206, 1998.
			\bibitem{steinbook} E.M. Stein, Harmonic Analysis: real-variable methods, orthogonality, and osciollatory integrals, Princeton University Press, Princeton, NJ, 1993. 
			
			\bibitem{Taylor} M. Taylor, Pesudodifferential operators and nonlinear PDE. Progress in Mathematics, 100. Birkh\"auser Boston, Inc., Boston, MA, 1991.
			\bibitem{taylor}M. Taylor,  Commutator estimates. Proc. Amer. Math. Soc. 131 (2003), no. 5, 1501-1507.
			
			
  
			\bibitem{wz15} F. Wang and X.C. Zhang, Heat kernel for fractional diffusion operatoers with pertubations, Forum Math. 27 (2015), 973-994
			\bibitem{wangzhu}X. Wang and J. Zhu,  A lower bound for the nodal sets of Steklov eigenfunctions. Math. Res. Lett. 22 (2015), no. 4, 1243-1253	
			\bibitem{zo86}V. M. Zolotarev, One-dimensional stable distributions, Transl. Math.Monographs 65, Amer.Math.
			Soc., Providence, RI, 1986.
		\end{thebibliography}
		
	\end{document}